\documentclass{amsart}
\usepackage{amssymb, amsmath, latexsym, mathabx, dsfont}

\renewcommand{\baselinestretch}{\baselinestretch}
\renewcommand{\baselinestretch}{1.1}
\numberwithin{equation}{section}

\newtheorem{thm}{Theorem}[section]
\newtheorem{lem}[thm]{Lemma}

\theoremstyle{definition}

\theoremstyle{remark}
\newtheorem{rmk}[thm]{Remark}

\newcommand{\Mod}[1]{\ (\mathrm{mod}\ #1)}

\begin{document}

\title[Ternary universal sums of generalized polygonal numbers]{Ternary universal sums of generalized\\ polygonal numbers}

\author{Jangwon Ju,  Byeong-Kweon Oh, and Bangnam Seo}

\address{Department of Mathematics, University of Ulsan, Ulsan 44610, Republic of Korea}
\email{jangwonju@ulsan.ac.kr}
\thanks{The work of the first author was supported by Basic Science Research Program through the National Research Foundation of Korea(NRF) funded by the Ministry of Education(NRF-2018R1A6A3A01012662).}

\address{Department of Mathematical Sciences and Research Institute of Mathematics, Seoul National University, Seoul 08826, Republic of Korea}
\email{bkoh@snu.ac.kr}
\thanks{The work of the second author was supported by the National Research Foundation of Korea (NRF-2017R1A2B4003758).}

\address{Department of Mathematical Sciences, Seoul National University, Seoul 08826, Republic of Korea}
\email{bangnam1@snu.ac.kr}

\subjclass[2010]{Primary 11E12, 11E20}

\keywords{generalized polygonal numbers, ternary universal sums}

\begin{abstract} An integer of the form $p_m(x)= \frac{(m-2)x^2-(m-4)x}{2} \ (m\ge 3)$, for some integer $x$ is called a generalized polygonal number of order $m$.  A ternary sum 
$\Phi_{i,j,k}^{a,b,c}(x,y,z)=ap_{i+2}(x)+bp_{j+2}(y)+cp_{k+2}(z)$  of generalized polygonal numbers,  for some positive integers $a,b,c$ and some integers $1\leq i\leq j \leq k$,  is said to be {\it universal over $\mathbb{Z}$} if the equation $\Phi_{i,j,k}^{a,b,c}(x,y,z)=n$ has an integer solution $x,y,z$ for any nonnegative integer $n$.  In this article, we prove the universalities of $17$ ternary sums  of generalized polygonal numbers, which was conjectured by Sun.  
\end{abstract}

\maketitle

\section{Introduction} 
 A polygonal number is a number represented as the number of dots and pebbles arranged in the shape of a regular polygon. More precisely,   {\it a polygonal number of order $m$} (or {\it an $m$-gonal number}) for $m\ge3$ is defined by
$$
p_m(x)= \frac{(m-2)x^2-(m-4)x}{2}
$$
for some nonnegative integer $x$. If we allow $x$ to be negative, then $p_m(x)$ is called  
{\it a generalized polygonal number of order $m$} (or {\it a generalized $m$-gonal number}). 
By definition, every polygonal number of order $m$ is a generalized polygonal number of order $m$. However, the converse is  true only when $m=3,4$. 

A famous assertion of Fermat states that every positive integer can be expressed as a  sum of three triangular numbers, which was proved by  Gauss in 1796. 
This follows from the Gauss-Legendre theorem, which states that every positive integer which is not of the form $4^k(8l+7)$ with  nonnegative integers $k$ and $l$, is a sum of three squares of integers.  

In general, a ternary sum 
$$
\Phi_{i,j,k}^{a,b,c}=ap_{i+2}(x)+bp_{j+2}(y)+cp_{k+2}(z) \ (a,b,c >0,~1\leq i\leq j\leq k)
$$ 
of polygonal numbers is called {\it universal over $\mathbb N$} if the diophantine equation
\begin{equation} \label{fund}
ap_{i+2}(x)+bp_{j+2}(y)+cp_{k+2}(z)=n
\end{equation}
has  a {\it nonnegative} integer solution $x,y,z$ for any nonnegative integer $n$. It is called {\it universal over $\mathbb{Z}$} if the above equation has an integer solution for any nonnegative integer $n$. 

In 1862, Liouville generalized Gauss's triangular theorem as follows:
 for positive integers $a,b,c \ (a \leq b \leq c)$, the ternary sum $ap_3(x)+bp_3(y)+cp_3(z)$ of triangular numbers is universal if and only if $(a,b,c)$ is one of  the following triples:
$$
\begin{array}{llllllll}
&(1,1,1), &(1,1,2), &(1,1,4), &(1,1,5), &(1,2,2), &(1,2,3), &(1,2,4).
\end{array}
$$
In \cite{gps}, \cite{os}, and \cite{s1}, Sun and his collaborators determined  all  ternary universal sums of the form $ap_3(x)+bp_3(y)+cp_4(z)$ or $ap_3(x)+bp_4(y)+cp_4(z)$.  There does not exist a universal sum of the form $ap_4(x)+bp_4(y)+cp_4(z)$ for  any positive integers $a,b,c$. 
Hence,  to find all ternary sums of polygonal numbers that are universal over $\mathbb N$, we may assume that $\{i,j,k\} \ne \{3,4\}$. Recently, in \cite{s2} Sun gave a complete list of $95$  candidates of  ternary  sums of polygonal numbers that are universal over $\mathbb N$ under this assumption. 

In fact, the proof of the universality over $\mathbb{Z}$ of a ternary sum of polygonal numbers is much easier than the proof of universality over $\mathbb N$.  For this reason, one may naturally ask the universalities over $\mathbb{Z}$ of $95$ candidates of ternary sums.  Related with this question, Sun and Ge proved  in \cite{fz} and \cite{s2} that $62$ ternary sums among $95$ candidates are, in fact, universal over $\mathbb{Z}$. 
The remaining $33=95-62$ candidates are  as follows:
\begin{equation} \label{list}
{\setlength\arraycolsep{2pt}
\begin{array}{lllll}
&{\bf p_3+9p_3+p_5},&{\bf p_3+2p_4+4p_5},&{\bf p_3+4p_4+2p_5},&{\bf p_3+2p_3+p_7},\\
&{\bf p_3+p_4+2p_7},&{\bf p_3+p_5+p_7},&p_3+p_5+4p_7,&{\bf p_3+2p_5+p_7},\\
&{\bf p_3+p_7+2p_7},&{\bf p_3+2p_3+2p_8},&{\bf p_3+p_7+p_8},&{\bf p_3+2p_3+p_9},\\
&p_3+2p_3+2p_9,&{\bf p_3+p_4+p_9},&p_3+p_4+2p_9,&p_3+2p_4+p_9,\\
&{\bf p_3+p_5+2p_9},&{\bf 2p_3+p_5+p_9},&{\bf p_3+p_3+p_{12}},&{\bf p_3+2p_3+p_{12}},\\
&p_3+2p_3+2p_{12},&p_3+p_4+p_{13},&{\bf p_3+p_5+p_{13}},&p_3+2p_3+p_{15},\\
&p_3+p_4+p_{15},&p_3+2p_3+p_{16},&p_3+p_3+p_{17},&p_3+2p_3+p_{17},\\
&p_3+p_4+p_{17},&p_3+2p_4+p_{17},&p_3+p_4+p_{18},&p_3+2p_3+p_{23},\\
&p_3+p_4+p_{27}.\\
\end{array}}
\end{equation}
In this paper, we prove that  $17$  ternary sums written in boldface in the above list of  remaining candidates are universal over $\mathbb Z$.  In most cases, we use the method developed in \cite{regular} and
\cite{pentagonal}.  For those who are  unfamiliar with it, we briefly review this method in Section 2.  

Let 
$$
f(x_1,x_2,\dots,x_n)=\sum_{1 \le i, j\le n} a_{ij} x_ix_j \ (a_{ij}=a_{ji})
$$
 be a positive definite integral quadratic form. 
The corresponding symmetric matrix of $f$ is defined by $M_f=(a_{ij})$. 
For a diagonal quadratic form 
$f(x_1,x_2,\dots,x_n)=a_1x_1^2+a_2x_2^2+\cdots+a_nx_n^2$, we simply write 
$$
M_f=\langle a_1,a_2,\dots,a_n\rangle.
$$ 
The discriminant of $f$ is defined by the determinant of $M_f$, and denoted by $df$. 
The genus of $f$ is the set of all quadratic forms that are locally isometric to $f$ over $\mathbb{Z}_p$ for any prime $p$, which is denoted by $\text{gen}(f)$.  
The subset of  $\text{gen}(f)$ consisting of all quadratic forms that are isometric to  $f$ itself is denoted by $[f]$.  
The number of isometry classes in $\text{gen}(f)$ is called the class number of $f$, and denoted by $h(f)$.
An integer $a$ is called {\it an eligible integer of $f$} if it is locally represented by $f$ over $\mathbb{Z}_p$ for any prime $p$. 
Note that any eligible integer of $f$ is represented by a quadratic form in the genus of $f$ (for details, see 102:5 of \cite{om}).    
The set of all eligible integers of $f$ is denoted by $Q(\text{gen}(f))$.  Similarly, $Q(f)$ denotes the set of all integers that are represented by $f$ itself.   
For an integer $a$, $R(a,f)$ denotes the set of all vectors $x\in \mathbb{Z}^n$ such that $f(x)=a$, and $r(a,f)$ denotes its cardinality. For a set $S$, we define $\pm S=\{ s : s \in S \ \text{or} \ -s \in S\}$. For two vectors $v=(v_1,v_2,\dots,v_n)$, $w=(w_1,w_2,\dots,w_n) \in \mathbb{Z}^n$ and a positive integer $s$, if $v_i\equiv w_i \Mod{s}$ for any $i=1,2,\dots,n$, then we write $v\equiv w \Mod{s}$.

Any unexplained notations and terminologies can be found in \cite{ki} or \cite{om}.

\section{General tools}
Let $i,j$, and  $k$ be integers with $1 \le i \le j \le k$. Recall that for positive integers $a,b$, and $c$, a ternary sum 
$$
\Phi_{i,j,k}^{a,b,c}(x,y,z)=ap_{i+2}(x)+bp_{j+2}(y)+cp_{k+2}(z)
$$
 of generalized polygonal numbers is universal  over $\mathbb{Z}$ if  $\Phi_{i,j,k}^{a,b,c}(x,y,z)=n$ 
has an integer solution for any nonnegative integer $n$.  Since any generalized hexagonal number is a triangular number and vice versa, we always assume that neither $i,j$ nor $k$  is $4$. 
Note that $\Phi_{i,j,k}^{a,b,c}$ is universal over $\mathbb{Z}$ if and only if the equation
\begin{equation} \label{im}
\begin{array} {rl}
&jka(2ix-(i-2))^2+ikb(2jy-(j-2))^2+ijc(2kz-(k-2))^2\\
&\hskip 1cm =8ijkn+jka(i-2)^2+ikb(j-2)^2+ijc(k-2)^2=: \Psi_{i,j,k}^{a,b,c}(n)
\end{array}
\end{equation}
has an integer solution for any nonnegative integer $n$.  This is equivalent to the existence of an integer solution $X,Y,Z$ of the diophantine equation
\begin{equation}\label{cong}
jkaX^2+ikbY^2+ijcZ^2=\Psi_{i,j,k}^{a,b,c}(n)
\end{equation}
 satisfying the following congruence conditions
 \begin{equation} \label{cong2}
X\equiv -i+2 \Mod{2i},~  Y\equiv -j+2 \Mod{2j},  \text{ and } Z\equiv -k+2 \Mod{2k}.
\end{equation} 

 In some particular cases, representations of quadratic forms with  some congruence conditions correspond to representations of a subform which is taken suitably (for details, see \cite{jko}).  The following lemma reveals this correspondence more concretely in our case.

\begin{lem}  \label{aaa}  Assume that $i,j,k$ and $a,b,c$ given above satisfy the followings:
\begin{itemize}
\item [(i)] $\{i,j,k\} \subset \mathcal S_P=\{ 1, 2, p^n, 2p^n : \text{$p$ is an odd prime and $n\ge1$}\}$;
\item [(ii)] all of $\text{gcd}(i,jka)$, $\text{gcd}(j,ikb)$, and $\text{gcd}(k,ijc)$ are either $1$ or $2$;
\item [(iii)] either $i,j,k$ are all even or some of them are odd and 
$$
\Psi_{i,j,k}^{a,b,c}(n) \not \equiv 0 \Mod{2^{3+\text{ord}_2(ijk)}}.
$$
\end{itemize}
Then there is a quadratic form $Q(X,Y,Z)$ such that for any integer $n$, $Q(X,Y,Z)=\Psi_{i,j,k}^{a,b,c}(n)$ has an integer solution if and only if Eq. \eqref{cong} has an integer solution satisfying \eqref{cong2}.
\end{lem} 
\begin{proof}  Note that for any $s \in \mathcal S_P$ and any integer $a$ with $\text{gcd}(a,s)=1$, the congruence equation $x^2\equiv a^2 \Mod{s}$ has at most two incongruent integer solutions modulo $s$. 
 
 First, assume that  all of $i,j$, and $k$ are even. From the assumption, we have $i \equiv j\equiv k \equiv 2 \Mod{4}$.  Define
$$
Q(X,Y,Z)=16(jkaX^2+ikbY^2+ijcZ^2).
$$
Assume that  $(u,v,w) \in \mathbb{Z}^3$ is an integer solution of $Q(X,Y,Z)=\Psi_{i,j,k}^{a,b,c}(n)$. Then 
$$
16jkau^2 \equiv  jka(i-2)^2 \Mod{32i}.  
$$ 
From this and the conditions (i) and (ii) follows $4u \equiv \pm(i-2) \Mod{2i}$. Similarly,  we have $4v \equiv \pm(j-2) \Mod{2j}$ and $4k \equiv \pm(k-2) \Mod{2k}$.  Therefore by choosing  the signs  suitably, 
we may find an integer solution of Eq. \eqref{cong} satisfying \eqref{cong2}.  For the converse, note that any integers $X,Y,Z$ satisfying the condition \eqref{cong2} are divisible by $4$.

Assume that exactly one of $i,j,k$, say $i$, is odd.  Define 
$$
Q(X,Y,Z)=jkaX^2+16(ikbY^2+ijcZ^2).
$$
Assume that  $(u,v,w) \in \mathbb{Z}^3$ is an integer solution of $Q(X,Y,Z)=\Psi_{i,j,k}^{a,b,c}(n)$.
Since 
$$
16ikbv^2 \equiv ikb(j-2)^2 \Mod{j} \quad \text{and} \quad 16ijcw^2\equiv ijc(k-2)^2 \Mod{k},
$$
and  $\text{gcd}(j, ikb)=\text{gcd}(k,ijc)=2$ from (ii), we have 
$$
16v^2\equiv (j-2)^2 \ \left(\text{mod } {\frac j2}\right) \quad \text{and}\quad 16w^2\equiv (k-2)^2 \ \left(\text{mod } {\frac k2}\right),
$$   
which implies that $4v \equiv \pm(j-2) \Mod{2j}$ and $4k \equiv \pm(k-2) \Mod{2k}$. Now since $jkau^2 \equiv jka(i-2)^2 \Mod{32i}$ and $a \not \equiv 0 \Mod{8}$ by the condition (iii), we have $u^2 \equiv (i-2)^2 \Mod{2i}$.  Hence $u \equiv \pm (i-2) \Mod{2i}$. Therefore by choosing the signs suitably, we may find an integer solution of \eqref{cong} satisfying \eqref{cong2}.

Assume that exactly two of $i,j,k$, say $j,k$, are odd.  Define 
$$
Q(X,Y,Z)=16jkaX^2+ikbY^2+ijc(Y-4Z)^2.
$$
Assume that  $(u,v,w) \in \mathbb{Z}^3$ is an integer solution of $Q(X,Y,Z)=\Psi_{i,j,k}^{a,b,c}(n)$.
Then, similarly to the above, we may easily show that
$$
4u\equiv \pm(i-2) \ \left(\text{mod } {\frac i2}\right), \ v\equiv \pm(j-2) \Mod{j}, \  v-4w \equiv \pm(k-2)\Mod{k}.
$$ 
Since $i \equiv 2\Mod{4}$, we have $4u\equiv \pm(i-2) \Mod{2i}$. Furthermore, since
$$
ikbv^2+ijc(v-4w)^2\equiv ikb+ijc  \not \equiv 0  \Mod{16},
$$
$v$ is odd. Hence 
$$
v\equiv \pm(j-2) \Mod{2j}, \  v-4w \equiv \pm(k-2) \Mod{2k}.
$$
Therefore by choosing the signs suitably, we may find  an integer solution of \eqref{cong} satisfying \eqref{cong2}.  Conversely, if Eq. \eqref{cong} satisfying \eqref{cong2} has an integer solution $X,Y,Z$, then $X$ is divisible by $4$ and both $Y$ and $Z$ are odd. Hence $Y\equiv Z\Mod{4}$ or $Y \equiv-Z \Mod{4}$. This implies that  $Q(X,Y,Z)=\Psi_{i,j,k}^{a,b,c}(n)$ has an integer solution.
 
Finally, if all of $i,j,k$ are odd, then we define
$$
Q(X,Y,Z)=jkaX^2+ikb(X-4Y)^2+ijc(X-4Z)^2.
$$
The proof of this case is quite similar to the above. 
\end{proof}

\begin{rmk} \label{goodcase}  Among $33$ candidates given in the introduction, those which do not satisfy the conditions in Lemma \ref{aaa} are the cases
$$
4, \ 7, \ 9,\ 14, \ 16, \ 24,\ \text{and} \ 27\sim 32.
$$
For $4$-th, $9$-th, and $14$-th cases, we prove their universalities by using some other methods. For details, see Section 3. 
\end{rmk}

In \cite{regular} and \cite{pentagonal}, we developed a method on determining whether or not integers in an arithmetic progression are represented by some particular ternary quadratic form. We briefly introduce this method for those who are unfamiliar with it. 
 
 Let $d$ be a positive integer and let $a$ be a nonnegative integer $(a\leq d)$. We define 
$$
S_{d,a}=\{dn+a \mid n \in \mathbb N \cup \{0\}\}.
$$
For integral ternary quadratic forms $f,g$, we define
$$
R(g,d,a)=\{v \in (\mathbb{Z}/d\mathbb{Z})^3 \mid vM_gv^t\equiv a\Mod{d} \}
$$
and
$$
R(f,g,d)=\{T\in M_3(\mathbb{Z}) \mid  T^tM_fT=d^2M_g \}.
$$
A coset (or, a vector in the coset) $v \in R(g,d,a)$ is said to be {\it good} with respect to $f,g,d, \text{ and }a$ if there is a $T\in R(f,g,d)$ such that $\frac1d \cdot vT^t \in \mathbb{Z}^3$.  The set of all good vectors in $R(g,d,a)$ is denoted by $R_f(g,d,a)$. Every vector in $R(g,d,a)- R_f(g,d,a)$ is said to be {\it bad}.   
If there does not exist a bad vector, we write  
$$
g\prec_{d,a} f.
$$
If $g\prec_{d,a} f$, then by Lemma 2.2 of \cite{regular},  we have 
$$
S_{d,a}\cap Q(g) \subset Q(f).
$$
Note that the converse is not true in general.

In general, if $d$ is large, then it is not easy to compute the set  $R(g,d,a)- R_f(g,d,a)$ of bad vectors exactly by hand.  A MAPLE based  computer program for this set  is available upon request to the authors.

\begin{thm}\label{maintool}
Under the same notation given above, assume that there is a partition $R(g,d,a)-R_f(g,d,a)=(P_1\cup \cdots \cup P_k) \cup (\widetilde{P_1}\cup \cdots \cup \widetilde{P_{k'}})$ satisfying the following properties: for each $P_i$, there is a  $T_i\in M_3(\mathbb Z)$ such that
\begin{enumerate}
\item[(i)] $\frac{1}{d}T_i$ has an infinite order;
\item[(ii)] $T_i^tM_gT_i=d^2M_g$;
\item[(iii)]  for any vector $v \in \mathbb Z^3$ such that  $v\Mod{d}\in P_i$, \ $\frac1d\cdot vT_i^t\in \mathbb Z^3$ \ and $\frac1d\cdot vT_i^t\Mod{d} \in P_i \cup R_f(g,d,a)$, 
\end{enumerate}
and for each $\widetilde{P_j}$, there is a $\widetilde{T_{j}}\in M_3(\mathbb Z)$ such that
\begin{enumerate}
\item[(iv)] $\widetilde{T_{j}}^t M_g \widetilde{T_{j}}=d^2 M_g$;
\item[(v)]  for any vector $v \in \mathbb{Z}^3$  such that $v\Mod{d} \in \widetilde{P_j}$, \ $\frac1d\cdot v\widetilde{T_{j}}^t \in \mathbb{Z}^3$ \ and  $\frac1d\cdot v\widetilde{T_{j}}^t \Mod{d} \in P_1\cup\cdots\cup P_k \cup R_f(g,d,a)$. 
\end{enumerate}
Then we have
$$
(S_{d,a} \cap Q(g)) - \bigcup_{i=1}^{k} g(z_i)\mathcal S\subset Q(f),
$$
where the vector $z_i$ is any integral primitive eigenvector of $T_i$,  $\mathcal S$ is the set of squares of integers, and $g(z_i)\mathcal S=\{g(z_i)n^2\mid n\in \mathbb{Z}\}$. 
\end{thm}

\begin{proof}
From the conditions (iv) and (v), we may assume that 
$$
R(g,d,a)-R_f(g,d,a)=P_1\cup \cdots \cup P_k.
$$ 
Hence the theorem follows directly from Corollary 2.2 in \cite{pentagonal}.
\end{proof}

\begin{lem}  \label{spinor}  Let $f$ be a integral ternary quadratic form with class number greater than $1$. Assume that the genus of $f$ contains only one spinor genus. 
If an integer $a$ is represented by all isometric classes in the genus of $f$ except at most only one isometric class, then for any positive integer $s>1$ with $\text{gcd}(s,2df)=1$, $as^2$ is represented by all quadratic forms in the genus of $f$.  
\end{lem}

\begin{proof} Without loss of generality, we may assume that there is a quadratic form, say $g$, that does not represent $a$. For any prime $p$ not dividing $2df$, the graph $Z(p,f)$ defined in \cite{bh} is connected. Hence there is a quadratic form $g'$ such that $g'$ is adjacent to $g$ and $a$ is represented by $g'$. Since $p^2g'$ is represented by $g$, $ap^2$ is represented by $g$. This completes the proof.   \end{proof}

\section{Ternary universal sums of polygonal numbers}

In this section, we prove that 17 ternary sums of generalized polygonal numbers among 33 candidates are, in fact,  universal over $\mathbb Z$. For the list of all candidates, see \eqref{list} in the introduction.

In the case (i), we will explain how our method works in detail. 
Since everything is quite similar to this for all the other cases, we briefly provide all parameters needed for computations, in the remaining cases.
\\
\\ 
\textbf{(i)} $p_3(x)+9p_3(y)+p_5(z)$. We have to show that for any nonnegative integer $n$, the diophantine equation
$$
3(2x-1)^2+27(2y-1)^2+(6z-1)^2=24n+31
$$
has an integer solution $x,y,z$ by \eqref{im}, which is equivalent to the diophantine equation
\begin{equation}\label{i1}
X^2+3Y^2+27Z^2=24n+31
\end{equation}
has an integer solution $X,Y,Z$ such that $X\equiv Y\equiv Z\equiv1\Mod{2}$ and $X\not\equiv0\Mod{3}$ for any nonnegative integer $n$.
Note that $X\not\equiv0\Mod{3}$ if $X,Y,Z$ is an integer solution of Eq. \eqref{i1}. 
Furthermore, if there is an integer solution $X,Y,Z$ such that $X\equiv Y\Mod{4}$, then $X\equiv Y\equiv Z\equiv1\Mod{2}$.
Therefore, it is enough to show that the equation
\begin{equation}\label{i2}
(X+3Y)^2+3(X-Y)^2+27Z^2=4X^2+12Y^2+27Z^2=24n+31.
\end{equation}
has an integer solution $X,Y,Z$ for any nonnegative integer $n$.  

Let $f$ be a ternary quadratic form such that $M_f=\langle 4,12, 27\rangle$. Since the class number of $f$ is four, it is hard to determine whether or not every positive integer of the form $24n+31$ is represented by $f$. To apply Theorem \ref{maintool} to our situation, we need a ternary quadratic form $g$ such that 
 \begin{itemize} 
 \item [(i)] $g$ represents all integers of the form  $24n+31$ for a nonnegative integer $n$;
 \item [(ii)] $24^2g$ is represented by $f$.
 \end{itemize}
One may   easily show that a quadratic form $g$ such that $M_g=\langle 3,4,12\rangle$ satisfies all conditions given above. Note that the class number of $g$ is one.  One may check that there are exactly $2304$ vectors in $R(g,24,7)$ and $336$ matrices in $R(f,g,24)$. Furthermore, all vectors in $R(g,24,7)$ are good vectors with respect to $f,g,24$, and $7$ except the following $128$ vectors:
$$
\aligned
\quad\quad\quad R(g,24,7)-R_f(g,24,7)=&\{(v_1,v_2,v_3)\in (\mathbb{Z}/24\mathbb{Z})^3 \mid v_1\equiv1,5\Mod6,\\
&v_3\equiv 0\Mod6, v_1\equiv\pm(v_2+v_3) \Mod{12}\}.
\endaligned
$$
Now, define
$$
\aligned
&\quad\quad\quad P_1=\{(v_1,v_2,v_3)\in R(g,24,7)-R_f(g,24,7) \mid v_1\equiv -(v_2+v_3)\Mod{12}\},\\
&\quad\quad\quad \widetilde{P_1}=\{(v_1,v_2,v_3)\in R(g,24,7)-R_f(g,24,7) \mid v_1\equiv v_2+v_3\Mod{12}\},
\endaligned
$$
\noindent and 
$$
T_1=\begin{pmatrix}0&24&24\\18&-6&18\\-6&-6&18\end{pmatrix},\quad \widetilde{T_1}=\begin{pmatrix}0&-24&-24\\18&6&-18\\-6&6&-18\end{pmatrix}.
$$
Then one may easily show that all conditions in Theorem \ref{maintool} are satisfied.
All computations were done by a computer program based on MAPLE.

Assume that $g(v)=24n+31$ for some $v\in\mathbb{Z}^3$. 
We show that $f$ also represents $24n+31$.
If $v\Mod{24}$ is a good vector with respect to $f,g,24$, and $7$, then there is a $T\in R(f,g,24)$ such that
$$
\frac{1}{24}\cdot vT^t\in\mathbb{Z}^3~\text{and}~f\left(\frac{1}{24}\cdot vT^t\right)=24n+31.
$$
As a sample, for $(1,25,2)\in\mathbb{Z}^3$, note that $(1,25,2)\equiv(1,1,2)\Mod{24}\in R_f(g,24,7)$. 
In fact, there is a $T=\begin{pmatrix}9&21&9\\9&-3&-15\\4&-4&12\end{pmatrix}\in R(f,g,24)$ such that
$$
f\left(\frac{1}{24}\cdot(1,25,2)\cdot T^t\right)=f(23,-4,-3)=g(1,25,2)= 2551.
$$
Assume that $v\Mod{24}$ is a bad vector with respect to $f,g,24$, and $7$.
If $v\Mod{24}\in\widetilde{P_1}$, then by (v) of Theorem \ref{maintool}, we know that
$$
\frac{1}{24}\cdot v\widetilde{T_1}^t\in\mathbb{Z}^3~\text{and}~\frac{1}{24}\cdot v\widetilde{T_1}^t\Mod{24}\in P_1\cup R_f(g,24,7).
$$
Hence we may assume that $v\Mod{24}\in P_1$. Then by (iii) of Theorem \ref{maintool}, we know  that
$$
\frac{1}{24}\cdot vT_1^t\in\mathbb{Z}^3~\text{and}~\frac{1}{24}\cdot vT_1^t\Mod{24}\in P_1\cup R_f(g,24,7).
$$
If $\frac{1}{24}\cdot vT_1^t\Mod{24}\in R_f(g,24,7)$, then we are done. 
Assume $\frac{1}{24}\cdot vT_1^t\Mod{24}\in P_1$. 
Then by (iii) of Theorem \ref{maintool}, we directly know that 
$$
\left(\frac{1}{24}\right)^2\cdot v(T_1^t)^2\in\mathbb{Z}^3~\text{and}~\left(\frac{1}{24}\right)^2\cdot v(T_1^t)^2\Mod{24}\in P_1\cup R_f(g,24,7).
$$
Now, inductively, either there is a positive integer $m$ such that
$$
\left(\frac{1}{24}\right)^m\cdot v\left(T_1^t\right)^m\in\mathbb{Z}^3~\text{and}~\left(\frac{1}{24}\right)^m\cdot v\left(T_1^t\right)^m\Mod{24}\in R_f(g,24,7),
$$ 
or for any positive integer $m$,
$$
\left(\frac{1}{24}\right)^m\cdot v\left(T_1^t\right)^m\in\mathbb{Z}^3~\text{and}~\left(\frac{1}{24}\right)^m\cdot v\left(T_1^t\right)^m\Mod{24}\in P_1.
$$
Since there are only finitely many integer solution of Eq. \eqref{i2} and $T_1$ has an infinite order, the latter is impossible unless $v$ is an eigenvector of $T_1$.
Note that $z_1=(1,-1,0)$ is an integral primitive eigenvector of $T_1$.  Hence if $24n+31$ is not of the form $Q(\pm t,\mp t, 0)=7t^2$ for some positive integer $t$, then it is also represented by $f$.  Assume that $24n+31=7t^2$ for a positive integer $t$. Then there is a prime $p>3$ dividing $t$. Since the class number of $g$ is one, we have
$$
r(7p^2,g)=4p+4-4\cdot\left(\frac{-7}{p}\right)>4
$$
by the Minkowski-Siegel formula (for details, see \cite{ya}). Note that 
$$
\frac{1}{24}\cdot(-1,-1,0)\cdot\widetilde{T_1}^t=(1,-1,0).
$$  
Hence there is a representation of $24n+31$ by $g$ which is neither contained in the eigenspace   of $T_1$ nor the inverse image of the eigenspace by $\frac{1}{24}\cdot\widetilde{T_1}^t$. Therefore,  
every integer of the form $24n+31$ for a nonnegative integer $n$  is represented by $f$. 
This completes the proof. 
\\
\\
\textbf{(ii)} $p_3(x)+2p_4(y)+4p_5(z)$. Let $f$ and $g$ be integral quadratic forms such that 
$$
M_f=\langle3, 4, 48\rangle \quad \text{and} \quad M_g=\langle3, 4, 12\rangle.
$$
Since $S_{24,7} \subset Q(\text{gen}(g))$ and $h(g)=1$, we have $S_{24,7} \subset Q(g)$. One may easily show that
$$
R(g,4,3)-R_f(g,4,3)=\pm\{(1,0,1),(1,2,3),(1,0,3),(1,2,1)\}.
$$
Let $P_1=\pm\{(1,0,1),(1,2,3)\}$, $\widetilde{P_1}=\pm\{(1,0,3),(1,2,1)\}$ and 
$$
T_1=\begin{pmatrix}0&-4&-4\\3&1&-3\\-1&1&-3\end{pmatrix},\quad \widetilde{T_{1}}=\begin{pmatrix}0&4&4\\3&-1&3\\-1&-1&3\end{pmatrix}.
$$
Then one may easily show that all conditions in Theorem \ref{maintool} are satisfied. 
Note that $z_1=(1,0,1)$ is an integral primitive eigenvector of $T_1$  and $Q(z_1)=15$, which is not of the form $24n+7$.    
Therefore we have $S_{24,7}\subset Q(f)$. This implies that the equation $3X^2 +4Y^2+48Z^2=24n+7$
has an integer solution $X,Y,Z$ for any nonnegative integer $n$. Since $X$ is odd and $Y$ is relatively prime to $6$, the equation
$$
3(2x+1)^2+48y^2+4(6z-1)^2=24n+7
$$ 
has an integer solution for any nonnegative integer $n$. Hence by  \eqref{im}, the ternary sum $p_3(x)+2p_4(y)+4p_5(z)$ is universal over $\mathbb Z$.
\\
\\
\textbf{(iii)} $p_3(x)+4p_4(y)+2p_5(z)$. Let $f,g$, and $g'$ be integral quadratic forms such that
$$
f=\langle2, 3, 96\rangle,\quad g=\begin{pmatrix}5&1&1\\1&11&-1\\1&-1&11\end{pmatrix},\quad \text{and} \quad g'=\langle2\rangle\perp\begin{pmatrix}12&6\\6&27\end{pmatrix}.
$$
Note that $\text{gen}(f)=\{[f],[g],[g']\}$ and $S_{24,5}\subset Q(\text{gen}(f))$.
One may easily show  that $g'\prec_{4,1} g$.
Furthermore, one may easily show that  
$$
{\setlength\arraycolsep{2pt}
\begin{array}{lllllrl}
R(g,12,5)-R_f(g,12,5) = &\pm\{\!\!&(3,1,7),&(3,2,8),&(3,4,10),&(3,5,11), &~\\
~&~&(3,7,1),&(3,8,2),&(3,10,4),&(3,11,5)&\!\!\}.
\end{array}}
$$
If we define $P_1=R(g,12,5)-R_f(g,12,5)$ and  
$T_1=\begin{pmatrix}8&10&-10\\-6&-3&-9\\2&-11&-1\end{pmatrix}$, then all  conditions in Theorem \ref{maintool} are satisfied.
Since  $z_1=(0,1,1)$ is an integral primitive eigenvector of $T_1$ is and $Q(z_1)=20$, which is not of the form $24n+5$,  we have  
\begin{equation}\label{iii}
S_{24,5}\cap Q(g)\subset Q(f).
\end{equation}
Recall that every integer of the form $24n+5$ for any nonnegative integer $n$ is represented by $f,g$ or $g'$.
Assume $24n+5$ is represented by $g$. Then $f$ represents it  by \eqref{iii}. 
Assume $24n+5$ is represented by $g'$. Then it is represented by $g$ since $g'\prec_{4,1}g$. 
Furthermore, by \eqref{iii}, it is also represented by $f$.
Therefore every integer of the form $24n+5$, for any nonnegative integer $n$ is represented by $f$.  
If the equation $2X^2+3Y^2+96Z^2=24n+5$ has an integer solution $X,Y,Z$, then $X$ is relatively prime to $6$ and $Y$ is odd. Hence the equation
$$
3(2x-1)^2+96y^2+2(6z-1)^2=24n+5
$$ 
has an integer solution for any nonnegative integer $n$. Therefore the ternary sum  $p_3(x)+4p_4(y)+2p_5(z)$ is universal over $\mathbb Z$.
\\
\\
\textbf{(iv)} $p_3(x)+2p_3(y)+p_7(z)$. First, we show that for any eligible integer $8n$ of  $$
F(X,Y,Z)=X^2+5Y^2+10Z^2,
$$ 
the equation $F(X,Y,Z)=8n$ has an integral solution 
$(X,Y,Z)=(a,b,c)$ with $a\equiv b\equiv c\equiv 1 \Mod{2}$, unless $n\equiv 0 \Mod{5}$. 
Since the class number of $F$ is one, $F(X,Y,Z)=8n$ has an integer solution $(X,Y,Z)=(a,b,c)$. One may easily check that all of $a,b$, and $c$ have same parity.
 Suppose that $a,b$, and $c$ are all even.  Then $a\equiv b\Mod{4}$ and there is another representation 
 $$
F\left(\frac{a+5b+10c}{4},\frac{-a+3b-2c}{4},\frac{-a-b+2c}{4}\right)=F(a,b,c)=8n.
$$  
Since there are only finitely many representations of $8n$ by $F$  and furthermore, $T=\frac14\begin{pmatrix}1&5&10\\-1&3&-2\\-1&-1&2\end{pmatrix}$ has an infinite order, there is a positive integer $m$ such that each  component of the vector 
$T^m(a,b,c)^t=(a_m,b_m,c_m)^t$  is odd, unless $(a,b,c)$ is an eigenvector of $T$.  
Note that $(0,\mp2,\pm1)$ are the integral primitive eigenvectors of $T$ and $F(0,\mp2,\pm1)=30$.

Note that any positive integer of the form $40n+24$ for some nonnegative integer $n$ is an eligible integer of  $F$. Hence the equation $F(X,Y,Z)=40n+24$ has an integer solution $X,Y,Z$ such that  $X\equiv Y\equiv Z \equiv 1 \Mod{2}$ and $X\equiv \pm 2\Mod{5}$, which implies that    
$$
5(2x-1)^2+10(2y-1)^2+(10z-3)^2=40n+24
$$ 
has an integer solution for any nonnegative integer $n$. Therefore by \eqref{im},  the ternary sum $p_3(x)+2p_3(y)+p_7(z)$ is universal over $\mathbb{Z}$.
\\
\\
\textbf{(v)} $p_3(x)+p_4+2p_7(z)$. Let $f$ and $g$ be integral quadratic forms such that
$$
M_f=\langle2, 5, 40\rangle \quad \text{and} \quad M_g=\langle5, 8, 10\rangle.
$$
Note that $\text{gen}(f)=\{[f],[g]\}$. 
One may easily compute that 
$$
{\setlength\arraycolsep{2pt}
\begin{array}{lllllrl}
R(g,8,7)-R_f(g,8,7)=&\pm\{\!\!&(1,0,1),&(1,2,5),&(1,4,1),&(1,6,5),&~\\
~&~&(3,0,3),&(3,2,7),&(3,4,3),&(3,6,7),&~\\
~&~&(1,0,7),&(1,2,3),&(1,4,7),&(1,6,3),&~\\
~&~&(3,0,5),&(3,2,1),&(3,4,5),&(3,6,1)&\!\!\}.
\end{array}}
$$
Let
{\setlength\arraycolsep{2pt}
\begin{eqnarray*}
P_1&=&\{(v_1,v_2,v_3)\in R(g,8,7)-R_f(g,8,7)\mid v_1\equiv v_3 \Mod{4}\},\\
P_2&=&\{(v_1,v_2,v_3)\in R(g,8,7)-R_f(g,8,7)\mid v_1+v_3\equiv0 \Mod{4}\},
\end{eqnarray*}}
and
$$
T_1=\begin{pmatrix}4&8&4\\-5&2&5\\2&-4&6\end{pmatrix},
\quad T_2=\begin{pmatrix}4&8&-4\\-5&2&-5\\-2&4&6\end{pmatrix}.
$$
Then one may easily show that all conditions in Theorem \ref{maintool} are satisfied. 
Note that $z_1=(1,0,1)$ and $z_2=(1,0,-1)$ are integral primitive eigenvectors of $T_1$ and $T_2$, respectively.   
Therefore  $(S_{8,7}\cap Q(g)) - 15\mathbb{Z}^2\subset Q(f)$. 
Since $S_{40,23}\subset Q(\text{gen}(f))$, we have $S_{40,23}\subset Q(f)$. 
Note that if the equation $2X^2+5Y^2+40Z^2=40n+23$ has an integer solution $X,Y,Z$, then $X,Y$ are odd and $X\equiv \pm2 \Mod{5}$.  Hence the equation
$$
5(2x+1)^2+40y^2+2(10z-3)^2=40n+23
$$
has an integer solution for any nonnegative integer $n$, which implies that the ternary sum  $p_3(x)+p_4+2p_7(z)$ is universal over $\mathbb Z$.
\\
\\
\textbf{(vi)} $p_3(x)+p_5(y)+p_7(z)$. Let $f$ and $g$ be integral quadratic forms such that
$$
M_f=\langle5\rangle\perp\begin{pmatrix}12&-6\\-6&18\end{pmatrix}\quad \text{and} \quad M_g=\langle2, 15, 30\rangle.
$$
Note that $\text{gen}(f)=\{[f],[g]\}$.  In this case, we have 
$$
R(g,4,3)-R_f(g,4,3)=\pm\{(0,1,0),(2,1,2)\}.
$$
 We define $P_1=R(g,4,3)-R_f(g,4,3)$ and 
$T_1=\begin{pmatrix}1&0&15\\0&-4&0\\-1&0&1\end{pmatrix}$.  
Then one may easily show that all conditions in Theorem \ref{maintool} are satisfied.
Note that  $z_1=(0,1,0)$ is an integral primitive eigenvector of $T_1$ is and $Q(0,1,0)=15$. Therefore, 
$$
(S_{4,3}\cap Q(g)) - 15\mathbb{Z}^2\subset Q(f).
$$ 
Since $S_{120,47}\subset Q(\text{gen}(f))$, 
we have $S_{120,47}\subset Q(f)$.
Note that if  the equation
$$
3(Z-2X)^2+5Y^2+15Z^2=12X^2+5Y^2+18Z^2-12ZX=120n+47
$$
has an integer solution $X,Y,Z$, then $Z-2X\equiv Y\equiv Z\equiv 1 \Mod{2}$, $Y\not\equiv0 \Mod{3}$ and $Z-X\equiv \pm 2\Mod{5}$. Hence the equation
$$
15(2x-1)^2+5(6y-1)^2+3(10z-3)^2=120n+47
$$
has an integer solution for any nonnegative integer $n$, which implies that the ternary sum  $p_3(x)+p_5(y)+p_7(z)$ is universal over $\mathbb Z$ by \eqref{im}. 
\\
\\
\textbf{(vii)} $p_3(x)+2p_5(y)+p_7(z)$. Let $f$ and $g$ be integral quadratic forms such that
$$
M_f=\langle5\rangle\perp\begin{pmatrix}9&3\\3&21\end{pmatrix}\quad\text{and}\quad M_g=\begin{pmatrix}6&3\\3&9\end{pmatrix}\perp\langle20\rangle.
$$
Note that $\text{gen}(f)=\{[f],[g]\}$. One may easily show by a direct computation that $g\prec_{15,11} f$.
Since $S_{60,26} \subset Q(\text{gen}(f))$, we have $S_{60,26}\subset Q(f)$. Now consider the equation
\begin{eqnarray*}
120n+52&=&3(Z-4X)^2+10Y^2+15Z^2\\ 
&=&48X^2+18Z^2+10Y^2-24XZ=f'(X,Y,Z).
\end{eqnarray*}
Since $f'$ is isometric to $2f$, the above equation has an integer solution $X,Y,Z$ for any nonnegative integer $n$. 
Furthermore, since $Z-4X\equiv Y\equiv Z\equiv 1 \Mod{2}$, $Y\not\equiv 0 \Mod{3}$, and $Z-4X\equiv \pm 2 \Mod{5}$, the equation
$$
15(2x-1)^2+10(6y-1)^2+3(10z-3)^2=120n+52
$$
has an integer solution for any nonnegative integer $n$. Therefore the ternary sum $p_3(x)+2p_5(y)+p_7(z)$ is universal over $\mathbb Z$.
\\
\\
\textbf{(viii)} $p_3(x)+p_7(y)+2p_7(z)$. Let $f$ and $g$ be integral quadratic forms such that 
$$
M_f=\begin{pmatrix}3&0&1\\0&5&0\\1&0&17\end{pmatrix}\quad\text{and}\quad M_g=\langle2, 5, 25\rangle.
$$
Note that $\text{gen}(f)=\{[f],[g]\}$. One may easily show by  a direct computation that $g\prec_{8,0}f$.
Since $S_{40,32}\subset Q(\text{gen}(f))$, we have  $S_{40,32}\subset Q(f)$.
Now consider the equation
\begin{eqnarray*}
40n+32&=&(Y-5X)^2+2Y^2+5Z^2\\
&=&25X^2+3Y^2+5Z^2-10XY=f'(X,Y,Z).
\end{eqnarray*}
Since $f'$ is isometric to $f$, the above equation has an integer solution for any nonnegative integer $n$.
Therefore the equation $f''(X,Y,Z)=X^2+2Y^2+5Z^2=40n+32$ has an integer solution $X,Y,Z$ with $X\equiv Y\Mod{5}$. Furthermore all of $X,Y$, and $Z$ have same parity. 
Suppose that  $(x,y,z)=(a,b,c)$  is an integer solution of  $f''(X,Y,Z)=40n+32$ such that  $a\equiv b \Mod{5}$ and $a\equiv b\equiv c\equiv 0 \Mod{2}$. Then $a\equiv c \Mod{4}$ and there is another representation 
$$
f''\left(\frac{-3a+2b+5c}{4},\frac{a-2b+5c}{4},\frac{-a-2b-c}{4}\right)=f''(a,b,c)=40n+32
$$
 satisfying
$$
\frac{-3a+2b+5c}{4} \equiv \frac{a-2b+5c}{4} \Mod{5}.
$$
Since $T=\frac14\begin{pmatrix}-3&2&5\\1&-2&5\\-1&-2&-1\end{pmatrix}$  has an infinite order, there is a positive integer $m$ such that each component of the vector 
$T^m(a,b,c)^t=(a_m,b_m,c_m)^t$ is odd and $a_m \equiv b_m \Mod{5}$, unless $(a,b,c)$ is an eigenvector of $T$. Note that $(a,b,c)=(-2,1,0)$ is an integral primitive eigenvector of $T$ and
 $f''(-2t,t,0)=6t^2\not\equiv 2 \Mod{5}$.  Hence  for any nonnegative integer $n$, $f''(X,Y,Z)=40n+32$ has an integer solution $X,Y,Z$ such that  $X\equiv Y\Mod{5}$ and  all of $X,Y$, and $Z$ are odd. Therefore the equation
$$
5(2x-1)^2+(10y-3)^2+2(10z-3)^2=40n+32
$$
has an integer solution for any nonnegative integer $n$, which implies that  the ternary sum  $p_3(x)+p_7(y)+2p_7(z)$ is universal over $\mathbb Z$ by \eqref{im}.
\\
\\
\textbf{(ix)} $p_3(x)+2p_3(y)+2p_8(z)$. Let $f$ and $g$ be integral quadratic forms such that
$$
M_f=\langle3, 6, 16\rangle\quad\text{and}\quad M_g=\langle1\rangle\perp\begin{pmatrix}18&6\\6&18\end{pmatrix}.
$$
Note that $\text{gen}(f)=\{[f],[g]\}$.  By a direct computation, we  may easily check that  $R(g,3,1)-R_f(g,3,1)=\pm\{(1,0,0)\}$.
We define 
$$
P_1=\pm\{(1,0,0)\} \quad  \text{and} \quad T_1=\begin{pmatrix}3&0&0\\0&-2&-3\\0&3&0\end{pmatrix}.
$$ 
Then one may easily show that all conditions in Theorem \ref{maintool} are satisfied.
Note that $z_1=(1,0,0)$ is an integral primitive eigenvector of $T_1$ is  and $Q(1,0,0)=1$.  Since $S_{24,1} \subset Q(\text{gen}(f))$, we have $S_{24,1}  \subset Q(f)$, unless $24n+1$ is a square of an integer. 
Since $g$ is contained in the spinor genus of $f$ and $1$ is represented by $g$, every square of an integer that has a prime divisor bigger than $3$  is represented by both $f$ and $g$ by Lemma  \ref{spinor}. 
Therefore every integer of the form $24n+1$ for some positive integer $n$ is represented by $f$. 
If  the equation $3X^2+6Y^2+16Z^2=24n+25$ has an integer solution $X,Y,Z$,  then $X,Y$ are odd and $Z$ is not divisible by 3. Hence the equation
$$
3(2x-1)^2+6(2y-1)^2+16(3z-1)^2=24n+25
$$ 
 has an integer solution for any nonnegative integer $n$, which implies that the ternary sum  $p_3(x)+2p_3(y)+2p_8(z)$ is universal over $\mathbb Z$ by \eqref{im}.
\\
\\
\textbf{(x)} $p_3(x)+p_7(y)+p_8(z)$. Let $f$ and $g$ be integral quadratic forms such that
$$
M_f=\begin{pmatrix}6&3\\3&9\end{pmatrix}\perp\langle20\rangle\quad\text{and} \quad M_g=\langle5\rangle\perp\begin{pmatrix}9&3\\3&21\end{pmatrix}.
$$
Note that $\text {gen}(f)=\{[f],[g]\}$. One may easily compute that $g\prec_{30,11}f$. Since
$S_{60,41}\subset Q(\text{gen}(f))$, we  have $S_{60,41}\subset Q(f)$.
Hence the equation 
$$
120n+82=3(Y-2X)^2+15Y^2+40Z^2=12X^2+18Y^2+40Z^2-12XY
$$
has an integer solution $X,Y,Z$ for any nonnegative integer $n$. Since $Y-2X\equiv Y\equiv 1\Mod{2}$, $Y-2X\equiv \pm 2\Mod{5}$, and $Z\not\equiv 0 \Mod{3}$, the equation
$$
15(2x-1)^2+3(10y-3)^2+40(3z-1)^2=120n+82 
$$
has an integer solution for any nonnegative integer $n$. Therefore by \eqref{im}, the ternary sum  $p_3(x)+p_7(y)+p_8(z)$ is universal over $\mathbb Z$.
\\
\\
\textbf{(xi)} $p_3(x)+2p_3(y)+p_9(z)$. Let $f$ and $g$ be integral quadratic forms such that
$$
M_f=\langle1\rangle\perp\begin{pmatrix}21&14\\14&28\end{pmatrix},\quad M_g=\langle1, 7, 14\rangle,\quad{and}\quad M_{g'}=\langle2, 7, 7\rangle.
$$
Note that $\text{gen}(g)=\{[g],[g']\}$. One may easily show by a direct computation that $g'\prec_{8,6} g$. Since 
$S_{56,46}\subset Q(\text{gen}(g))$, we have $S_{56,46}\subset Q(g)$.
 Now, one may show by a direct computation that  
 $$
 R(g,8,6)-R_f(g,8,6)=\pm\{(0,0,1),(0,0,3),(4,4,1),(4,4,3)\}.
 $$
We define $P_1= R(g,8,6)-R_f(g,8,6)$ and $T_1=\begin{pmatrix}1&21&0\\-3&1&0\\0&0&-8\end{pmatrix}$.  
Then one may easily show that all conditions in Theorem \ref{maintool} are satisfied.
Note that $z_1=(0,0,1)$ is an integral primitive eigenvector of $T_1$ and $Q(0,0,1)=14$.  
Since we know $(S_{8,6}\cap Q(g)) - 14\mathbb{Z}^2\subset Q(f)$ by Theorem \ref{maintool},  we have $S_{56,46} \subset Q(f)$.
Hence the equation
$$
56n+46=X^2+7(Z-2Y)^2+14Z^2=X^2+28Y^2+21Z^2-28YZ
$$
 has an integer solution $X,Y,Z$ for any nonnegative integer $n$.
Since $X\equiv Z-2Y\equiv Z\equiv 1\Mod{2}$ and $X\equiv \pm 2 \Mod{7}$, the equation
$$
7(2x-1)^2+14(2y-1)^2+(14z-5)^2=56n+46
$$
has an integer solution for any nonnegative integer $n$,  which implies that the ternary sum  $p_3(x)+2p_3(y)+p_9(z)$ is universal over $\mathbb Z$ by \eqref{im}.
\\
\\
\textbf{(xii)} $p_3(x)+p_4(y)+p_9(z)$. Let $f$ and $g$ be integral quadratic forms such that
$$
M_f=\begin{pmatrix}2&1\\1&4\end{pmatrix}\perp\langle28\rangle\quad\text{and}\quad 
M_g= \begin{pmatrix}4&2&2\\2&8&1\\2&1&8\end{pmatrix}.
$$
Note that $\text{gen}(f)=\{[f],[g]\}$. One may easily show by a direct computation that $g\prec_{4,0}f$. Since $S_{28,16}\subset Q(\text{gen}(f))$, we have $S_{28,16}\subset Q(f)$. 
 Hence the equation
$$
56n+32=(Y+2X)^2+7Y^2+56Z^2=4X^2+8Y^2+56Z^2+4XY
$$
has an integer solution for any nonnegative integer $n$. Therefore the equation $f'(X,Y,Z)=X^2+7Y^2+56Z^2=56n+32$ has an integer solution $a,b,c$ with $a\equiv b \Mod{2}$. Suppose that  $a,b,c$ is an integer solution of  $f'(x,y,z)=56n+32$  satisfying $a\equiv b\equiv 0 \Mod{2}$. Note that $a\equiv b\Mod4$ and 
$$ 
f'\left(\frac{3a-7b}{4},\frac{a+3b}{4},c\right)=f'(a,b,c).
$$
Since $T=\frac14 \begin{pmatrix}3&-7&0\\1&3&0\\0&0&4\end{pmatrix}$ has an infinite order,  there is  a positive integer $m$ such that  $a_m \equiv b_m \equiv 1 \Mod{2}$, where $T^m(a,b,c)^t=(a_m,b_m,c_m)^t$, unless $(a,b,c)$ is an eigenvector of $T$.
 Since $(0,0,1)$ is an integral primitive eigenvector of $T$ and $g(0,0,t)=56t^2\equiv 0\Mod{7}$, the equation $f'(X,Y,Z)=56n+32$ has an integer solution $X,Y,Z$ such that $X\equiv Y\equiv 1 \Mod{2}$ and $X\equiv \pm 2\Mod{7}$, for any nonnegative integer $n$. Hence the equation
$$
7(2x-1)+56y^2+(14z-5)^2=56n+32
$$
has an integer solution for any nonnegative integer $n$. Therefore by \eqref{im}, the ternary sum $p_3(x)+p_4(y)+p_9(z)$ is universal over $\mathbb Z$.
\\
\\
\textbf{(xiii)} $p_3(x)+p_5(y)+2p_{9}(z)$.  This case is one of the most complicate one. Let $f,g$, and $g'$ be integral quadratic forms such that
$$
M_f=\begin{pmatrix}13&1\\1&13\end{pmatrix}\perp\langle21\rangle,\quad
M_g=\begin{pmatrix}10&5&2\\5&13&1\\2&1&34\end{pmatrix},\quad\text{and}\quad
M_{g'}=\begin{pmatrix}10&4&-3\\4&10&3\\-3&3&45\end{pmatrix}.
$$
Note that $\text{gen}(f)=\{[f],[g],[g']\}$ and $S_{168,10} \subset Q(\text{gen}(f))$. One may easily show that  $g'\prec_{21,10}g$. 
Furthermore, One may easily show by  a direct computation that the set $R(g,21,10)-R_f(g,21,10)$ consists of $48$ vectors, which is the union of the following sets:
{\setlength\arraycolsep{2pt}
\begin{eqnarray*}
P_1&=&\pm \{ (4,0,17),  (10,0,11) \},\\
P_2&=&\pm\{ (2,17,19), (5,11,16) \},\\
P_3&=&\pm \{ (2,11,13), (5,17,1) \},\\
P_4&=&\pm\{ (1,0,0), (1,12,9), (8,0,0), (8,12,9) \},\\
\widetilde{P_1}&=&\pm \{ (2,9,1),(2,17,2),(4,12,8),(5,11,5),\\
& &\hskip 3.15cm(5,12,13),(10,9,20)\},\\
\widetilde{P_2}&=&\pm \{ (1,0,18), (1,4,1), (1,12,12), (2,0,13),\\
& &\hskip 3.15cm (5,0,1), (8,11,8), (8,12,12), (13,0,3) \}.\\
\end{eqnarray*}}
We also define
$$
\begin{small}
T_1=\begin{pmatrix}-14&-14&-35\\14&-7&14\\-7&-7&14\end{pmatrix},~
T_2=\begin{pmatrix}-21&-7&14\\6&-17&-2\\6&4&19\end{pmatrix},~
T_3=\begin{pmatrix}-15&-1&-34\\-6&1&34\\6&13&1\end{pmatrix},
\end{small}
$$
and
$$
\begin{small}
T_4=\begin{pmatrix}-21&-21&0\\0&19&-16\\0&5&19\end{pmatrix}, \qquad
T_5=\begin{pmatrix}-21&-21&0\\6&15&30\\6&-6&9\end{pmatrix}.
\end{small}
$$

Note that each $T_i$ satisfies the condition (i),(ii) of Theorem \ref{maintool}.
One may easily show that $P_1$ and $T_1$ satisfy the condition (iii) of Theorem \ref{maintool}.
For any $v\in\mathbb{Z}^3$ such that  $v\Mod{21}\in P_2$, One may show by a direct computation that 
$\frac{1}{21}\cdot vT_2^t\in\mathbb{Z}^3$ and if $\frac{1}{21}\cdot vT_2^t\Mod{21} \not \in R_f(g,21,10)$, then 
$$
\left(\frac{1}{21}\right)^2\cdot vT_2^tT_1^t\in\mathbb{Z}^3\quad  \text{and}  \quad \left(\frac{1}{21}\right)^2\cdot vT_2^tT_1^t \Mod{21} \in R_f(g,21,10)\cup P_2.
$$
For example, assume that $v=(v_1,v_2,v_3)\equiv (2,17,19) \Mod{21}$. Then $\frac{1}{21}\cdot vT_2^t$ is either a good vector or congruent to $(19,4,19)$ modulo $21$. If the latter holds, then 
$\left(\frac{1}{21}\right)^2\cdot vT_2^tT_1^t$ is either a good vector or congruent to $(2,17,19)$ modulo $21$. Note that the order of $T_1T_2$ is infinite. 

For any $v'\in\mathbb{Z}^3$ such that $v'\Mod{21}\in P_3$,  One may show by a direct computation that $\frac{1}{21}\cdot v'T_3^t\in\mathbb{Z}^3$ and if $\frac{1}{21}\cdot v'T_3^t\Mod{21}\not\in R_f(g,21,10) \cup P_2$, then 
$$
\left(\frac{1}{21}\right)^2\cdot v'T_3^tT_1^t\in\mathbb{Z}^3 \quad \text{and} \quad \left(\frac{1}{21}\right)^2\cdot v'T_3^tT_1^t\Mod{21}\in R_f(g,21,10)\cup P_3.
$$  
Note that the order of $T_1 T_3$ is infinite. 
Furthermore, for any $v''\in\mathbb{Z}^3$ such that $v''\Mod{21}\in P_4$, one may easily check  that 
$$
\frac{1}{21}\cdot v''T_4^t\in\mathbb{Z}^3 \quad \text{and} \quad \frac{1}{21}\cdot v''T_4^t\Mod{21}\in  R_f(g,21,10)\cup P_4 \cup \widetilde{P_2}.
$$
Note that  for any $\widetilde{v_1}\in\mathbb{Z}^3$ such that $\widetilde{v_1}\Mod{21}\in \widetilde{P_1}$, 
$$
\frac{1}{21}\cdot \widetilde{v_1}T_1^t \in\mathbb{Z}^3 \quad \text{and} \quad \frac{1}{21}\cdot\widetilde{v_1}T_1^t\Mod{21}\in R_f(g,21,10)\cup P_2.
$$
For $\widetilde{v_2}\in\mathbb{Z}^3$ such that $\widetilde{v_2}\Mod{21}\in\widetilde{P_2}$, then one of vectors in 
 $$
 \left\{\frac{1}{21}\cdot\widetilde{v_2}T_1^t,~\frac{1}{21}\cdot\widetilde{v_2} T_5^t,~\left(\frac{1}{21}\right)^2\cdot\widetilde{v_2}T_5^tT_1^t, ~
\left(\frac{1}{21}\right)^2\cdot\widetilde{v_2}(T_1^t)^2, ~\left(\frac{1}{21}\right)^3\cdot\widetilde{v_2}T_5^t(T_1^t)^2\right\}
 $$ 
 is integral  and contained in $R_f(g,21,10)\cup P_3$ modulo $21$.  
Therefore, though the situation is slightly different from the theorem, we may still apply Theorem \ref{maintool} to this case. 

Note that $(-1,0,1),(-1,2,1),(5,-4,1)$, and $(1,0,0)$ are integral primitive eigenvectors of $T_1,T_1T_2,T_1T_3$, and $T_4$, respectively. Note that 
$$
Q(-1,0,1)=40, \  \  Q(-1,2,1)=76,  \  \  Q(5,-4,1)=304, \quad \text{and} \quad Q(1,0,0)=10.
$$ 
Therefore if $168n+10$ is not of the form $10t^2$ for some positive integer $t$, then $168n+10$ is represented by $f$.  
Assume that $168n+10=10t^2$ for some positive integer $t>1$.  Then $t$ has a prime divisor relatively prime to $2\cdot3\cdot7$.  
 Since $g,g'\in\text{spn}(f)$,  and both $g$ and $g'$ represent $10$, $f$ represents $10t^2$ by Lemma \ref{spinor}.
Therefore every integer of the form $168n+10$ for a positive integer $n$ is represented by $f$. Hence the equation
$$
168n+178=6X^2+7(2Y-X)^2+21Z^2=13X^2+28Y^2+21Z^2-28XY
$$
 has an integer solution $X,Y,Z$ for any nonnegative integer $n$.
Since $X\equiv 2Y-X\equiv Z\equiv 1 \Mod{2}$, $2Y-X\not\equiv0 \Mod{3}$, and $X\equiv \pm2 \Mod{7}$, the equation
$$
21(2x-1)^2+7(6y-1)^2+6(14z-5)^2=168n+178
$$
has an integer solution for any nonnegative integer $n$, which implies that  the ternary sum $p_3(x)+p_5(y)+2p_{9}(z)$ is universal over $\mathbb Z$.
\\
\\
\textbf{(xiv)} $2p_3(x)+p_5(y)+p_9(z)$. Let $f,g$, and $g'$ be integral quadratic forms such that
$$
M_f=\begin{pmatrix}5&1\\1&17\end{pmatrix}\perp\langle21\rangle, \quad M_g=\begin{pmatrix}12&0&6\\0&14&7\\6&7&17\end{pmatrix}, \quad\text{and}\quad M_{g'}=\begin{pmatrix}5&2\\2&5\end{pmatrix}\perp\langle84\rangle.
$$
Note that $\text{gen}(f)=\{[f],[g],[g']\}$. One may easily show by a direct computation that $g\prec_{21,20}f$ and $g'\prec_{21,20}f$. Since $S_{84,62}\subset Q(\text{gen}(f))$, we have $S_{84,62}\subset Q(f)$.
Now consider the equation
\begin{eqnarray*}
168n+124&=&3(Y-4X)^2+7Y^2+42Z^2\\
&=&48X^2+10Y^2+42Z^2-24XY=f'(X,Y,Z).
\end{eqnarray*}
Since $f'$ is isometric to $2f$, the above equation has an integer solution $X,Y,Z$ for any nonnegative integer $n$. Since $Y\equiv Y-4X \Mod{4}$, $Y-4X$, $Y$, and $Z$ are all odd. Furthermore, $Y-4X\equiv \pm 2\Mod{7}$ and $Y\not\equiv 0 \Mod{3}$. Therefore the equation
$$
42(2x-1)^2+7(6y-1)^2+3(14z-5)^2=168n+124
$$
has an integer solution for any nonnegative integer $n$, which implies that the ternary sum $2p_3(x)+p_5(y)+p_9(z)$ is universal over $\mathbb Z$.
\\
\\
\textbf{(xv)} $p_3(x)+p_3(y)+p_{12}(z)$.  Let $f$ and $g$ be integral quadratic forms such that
$$
M_f=\langle4, 5, 5\rangle\quad\text{and}\quad M_g=\langle1, 5, 20\rangle.
$$
Note that $\text{gen}(f)=\{[f],[g]\}$ and $S_{20,1} \subset Q(\text{gen}(f))$.  
One may easily compute that 
$$
{\setlength\arraycolsep{2pt}
\begin{array}{lllllrl}
R(g,20,1)-R_f(g,20,1)=&\pm\{\!\!&(1,0,0),&(1,0,5),&(1,0,10),&(1,0,15),&\\ 
~&~&(1,10,0),&(1,10,5),&(1,10,10),&(1,10,15),&\\
~&~&(9,0,0),&(9,0,5),&(9,0,10),&(9,0,15),&\\
~&~&(9,10,0),&(9,10,5),&(9,10,10),&(9,10,15)&\!\!\}.
\end{array}}
$$
We define $P_1=R(g,20,1)-R_f(g,20,1)$ and 
$T_1=\begin{pmatrix}20&0&0\\0&-12&32\\0&-8&-12\end{pmatrix}$. 
Then one may easily show that all conditions in Theorem \ref{maintool} are satisfied.
Note that  $z_1=(1,0,0)$ is an integral primitive eigenvector of $T_1$ and $Q(1,0,0)=1$. Therefore if $20n+1$ is not of the form $t^2$ for a positive integer $t$, then it is represented by $f$.
Since $g$ is contained in the spinor genus of $f$ and $1$ is represented by $g$, every square of an integer that has a prime divisor relatively prime to $10$  is represented by $f$ by Lemma \ref{spinor}.  Therefore every integer of the form $20n+1$ for a positive integer $n$ is represented by $f$. Now consider the equation 
\begin{eqnarray*}
40n+42&=&5(Y-2X)^2+5Y^2+8Z^2\\
&=&20X^2+10Y^2+8Z^2-20XY=f'(X,Y,Z).
\end{eqnarray*}
Since $f'$ is isometric to $2f$, the above equation has an integer solution $X,Y,Z$ for any nonnegative integer $n$.
Since $Y-2X\equiv Y\equiv 1 \Mod{2}$ and $Z\equiv\pm2\Mod{5}$, the equation
$$
5(2x-1)^2+5(2y-1)^2+8(5z-2)^2=40n+42
$$
has an integer solution for any nonnegative integer $n$, which implies that the ternary sum $p_3(x)+p_3(y)+p_{12}(z)$ is universal over $\mathbb Z$.
\\
\\
\textbf{(xvi)} $p_3(x)+2p_3(y)+p_{12}(z)$. Let $f$ and $g$ be integral quadratic forms such that 
$$
M_f=\langle5, 8, 10\rangle\quad\text{and}\quad M_g=\langle2, 5, 40\rangle.
$$
Note that $\text{gen}(f)=\{[f],[g]\}$ and  $S_{40,7} \subset Q(\text{gen}(f))$.   One may show by a direct computation that 
$$
{\setlength\arraycolsep{2pt}
\begin{array}{lllllrl}
R(g,8,7)-R_f(g,8,7)=&\pm\{\!\!&(1,1,0),&(1,1,4),&(1,5,2),&(1,5,6),&\\
~&~&(3,3,0),&(3,3,4),&(3,7,2),&(3,7,6),&\\
~&~&(1,3,2),&(1,3,6),&(1,7,0),&(1,7,4),&\\
 ~&~&(3,1,2),&(3,1,6),&(3,5,0),&(3,5,4)&\!\!\}.
\end{array}}
$$
Let
{\setlength\arraycolsep{2pt}
\begin{eqnarray*}
P_1&=&\{(v_1,v_2,v_3)\in R(g,8,7)-R_f(g,8,7)\mid v_1-v_2+2v_3\equiv0 \Mod{8}\},\\
P_2&=&\{(v_1,v_2,v_3)\in R(g,8,7)-R_f(g,8,7)\mid v_1+v_2+2v_3\equiv0 \Mod{8}\},
\end{eqnarray*}}
and
$$
T_1=\begin{pmatrix}2&-10&20\\-4&-4&-8\\-1&1&6\end{pmatrix}, \quad
T_2=\begin{pmatrix}2&10&20\\4&-4&8\\-1&-1&6\end{pmatrix}.
$$
Then one may easily show that all conditions in Theorem \ref{maintool} are satisfied.
Note that $z_1=(1,1,0)$ and $z_2=(1,-1,0)$ are integral primitive eigenvectors for $T_1, T_2$, respectively and $Q(1,1,0)=Q(1,-1,0)=7$. 
Therefore if $40n+7$ is not of the form $7t^2$ for some positive integer $t$, then $40n+7$ is represented by $f$.

Assume that there is a positive integer $t>1$ such that $40n+7=7t^2$.  Then $t$ has a prime divisor relatively prime to $10$. 
Since $g\in\text{spn}(f)$  and $g$ represents $7$, $f$ represents $7t^2$ for any $t>1$ by Lemma \ref{spinor}.
Therefore every integer of the form $40n+7$ for a positive integer $n$ is represented by $f$. Hence the equation $5X^2+8Y^2+10Z^2=40n+47$ has an integer solution $X,Y,Z$ for any nonnegative integer $n$. Since $X\equiv Z\equiv 1 \Mod{2}$ and $Y\equiv \pm2\Mod{5}$, the equation
$$
5(2x-1)^2+10(2y-1)^2+8(5z-2)^2=40n+47
$$
has an integer solution for any nonnegative integer $n$. Therefore the ternary sum  $p_3(x)+2p_3(y)+p_{12}(z)$ is universal over $\mathbb Z$ by \eqref{im}.
\\
\\
\textbf{(xvii)} $p_3(x)+p_5(y)+p_{13}(z)$.  Let $f$ and $g$ be integral quadratic forms such that
$$
M_f=\begin{pmatrix}12&0&6\\0&44&22\\6&22&47\end{pmatrix},\quad M_g=\begin{pmatrix}15&6&3\\6&20&10\\3&10&71\end{pmatrix},\quad\text{and}\quad
M_{g'}=\begin{pmatrix}23&1&6\\1&23&6\\6&6&36\end{pmatrix}.
$$
Note that $\text{gen}(f)=\{[f],[g],[g']\}$ and $S_{264,23} \subset Q(\text{gen}(f))$. One may easily show by a direct computation that $g'\prec_{12,11}g$. 
Furthermore, one may easily compute that
$$
{\setlength\arraycolsep{2pt}
\begin{array}{lllllrl}
R(g,12,11)-R_f(g,12,11)=&\pm\{\!\!&(1,11,0),&(5,7,0),&(2,1,11),&(2,7,5),&\\
~&~&(4,5,1),&(4,11,7),&(3,2,0),&(3,10,0),&\\
~&~&(4,6,5),&(4,6,11),&(2,0,1),&(2,0,7)&\!\!\}.
\end{array}}
$$
We define a partition of $R(g,12,11)-R_f(g,12,11)$ as follows: 
\begin{eqnarray*}
P_1&=&\{(v_1,v_2,v_3)\mid v_1+v_2+v_3\equiv0 \Mod{12} \},\\
P_2&=&\{(v_1,v_2,v_3)\mid v_1+v_2+v_3\equiv2,10 \Mod{12}\},\\
\widetilde{P_1}&=&\{(v_1,v_2,v_3)\mid v_1+v_2+v_3\equiv1 \Mod{2}, \  \  v_2 \equiv 2 \Mod{4}\},\\
\widetilde{P_2}&=&\{(v_1,v_2,v_3)\mid v_1+v_2+v_3\equiv1 \Mod{2}, \  \  v_2 \equiv 0 \Mod{4}\}.
\end{eqnarray*}
We also define
$$
\begin{small}
T_1=\begin{pmatrix}-2&-14&-14\\-10&2&2\\2&2&-10\end{pmatrix},~
T_2=\begin{pmatrix}-8&-14&-6\\4&4&24\\4&-2&-6\end{pmatrix},~
\widetilde{T_1}=\begin{pmatrix}-12&-6&-12\\0&6&24\\0&-6&0\end{pmatrix},
\end{small}
$$
and $\widetilde{T_2}=\frac{1}{12}\cdot\widetilde{T_1}^2$.
Then one may easily show that all conditions in Theorem \ref{maintool} are satisfied.
Note that  $z_1=(-1,1,0)$ , $z_2=(-2,-1,1)$ are integral primitive eigenvectors for $T_1, T_2$, respectively,  and $Q(z_1)=23$, $Q(z_2)=143$.
Therefore if $264n+23$ is not of the form $23t^2$ for a positive integer $t$, then $264n+23$ is represented by $f$.

Assume that $264n+23=23t^2$ for some positive integer $t$ greater than $1$. Then $t$ has a prime divisor relatively prime to $2\cdot3\cdot11$.
Since $g,g'\in\text{spn}(f)$  and both $g$ and $g'$ represent $23$, $f$ represents $23t^2$ by Lemma \ref{spinor}.
Therefore every integer of the form $264n+23$ for a positive integer $n$ is represented by $f$. Hence the equation
$$
264n+287=3(Z-2X)^2+11(Z-2Y)^2+33Z^2=12X^2+44Y^2+47Z^2-12XZ-44YZ
$$
 has an integer solution $X,Y,Z$ for any nonnegative integer $n$. Since $Z-2X\equiv Z-2Y\equiv  Z\equiv 1\Mod{2}$, $Z-2X\equiv \pm2\Mod{11}$, and $Z-2Y\not\equiv0\Mod{3}$, the equation
$$
33(2x-1)^2+11(6y-1)^2+3(22z-9)^2=264n+287
$$ 
has an integer solution for any nonnegative integer $n$. Therefore the ternary sum $p_3(x)+p_5(y)+p_{13}(z)$ is universal over $\mathbb{Z}$ by \eqref{im}.

%
%

\end{document}